\RequirePackage{fix-cm}
\documentclass[smallextended]{svjour3}       
\smartqed  
\usepackage{amssymb,amsmath,multirow}
\usepackage{xcolor}

\newtheorem{thm}{Theorem}
\newtheorem{cor}{Corollary}
\newtheorem{exam}{Example}

\begin{document}
\title{Uniform Quadratic Optimization and Extensions
\thanks{This research was supported by National Natural Science Foundation of China under grant
11471325, by Beijing Higher Education Young Elite Teacher Project 29201442, and by the fund of State Key Laboratory of Software Development Environment under grants SKLSDE-2013ZX-13.}
}

\titlerunning{Uniform Quadratic Optimization and Extensions }        

\author{ Shu Wang  \and  Yong Xia}


\institute{
S. Wang \and Y. Xia  \at
             State Key Laboratory of Software Development
              Environment, LMIB of the Ministry of Education,
              School of
Mathematics and System Sciences, Beihang University, Beijing,
100191, P. R. China
              \email{wangshu.0130@163.com; ~dearyxia@gmail.com }
 }

\date{Received: date / Accepted: date}

\maketitle

\begin{abstract}
The uniform quadratic optimizatin problem (UQ) is a nonconvex quadratic
constrained quadratic programming (QCQP) sharing the same Hessian matrix. Based on the second-order cone programming (SOCP) relaxation, we establish a new
sufficient condition to guarantee strong duality for (UQ) and
then extend it to (QCQP), which not only covers several well-known results in literature but also partially gives answers to a few open questions. For convex constrained nonconvex (UQ), we propose an improved approximation algorithm based on (SOCP). Our approximation bound is  dimensional independent. As an application, we establish the first approximation bound for the problem of finding the Chebyshev center of the intersection of several balls.

\end{abstract}

\section{Introduction}
We study  the nonconvex uniform quadratic optimization problem:
\begin{eqnarray*}
{\rm (UQ)}~~&\max_{x\in \Bbb R^n}& f_{0}(x)\\
&{\rm s.~t.}& l_{i}\leq f_{i}(x)\leq u_{i},~i=1,\ldots,p,
\end{eqnarray*}
where $-\infty\leq l_{i}\leq u_{i}\leq +\infty$ and the quadratic functions $f_{i}(x)$ are defined by
\[
f_{i}(x)=x^TQx+2b_{i}^Tx+d_i,~i=0,1,\ldots,p,
\]
with $Q$ being a real symmetric  matrix, $b_{i}\in \Bbb R^n$ and $d_i\in \Bbb R$ for $i=0,1,\ldots,p$. Throughout this paper, we denote by (UQ$_+$) the widely used special case of (UQ) where $Q\succ0$ (a notation standing for that $Q$ is positive definite).

The problem (UQ$_+$) was first introduced by Beck \cite{Be07}
to find the smallest ball enclosing a given intersection of balls, or equivalently, find the Chebyshev center of the intersection of given balls, which can be reformulated as the following min-max problem:
\begin{equation}
{\rm (CC)}~~\min_{z}\max_{x\in \Omega}\|x-z\|^2=\mathop{\min_{z}}\left\{\max_{x\in \Omega}\left\{\|x\|^2-2x^Tz\right\}+\|z\|^2\right\}, \label{Cheyb}
\end{equation}
where $\Omega=\{x\in \Bbb R^n: \|x-a_{i}\|^2\leq r_{i}^2, i=1,\ldots, p\}$, $\|\cdot\|$ denotes the Euclidean norm, and the inner maximization problem is a special case of (UQ) with $Q$ being the identity matrix.
Moreover, we note that any NP-hard binary integer linear programming problem
\begin{eqnarray*}
&\max&~ c^Tx\\
&{\rm s.~t.}& a_i^Tx\leq b_i,~i=1,\ldots,m,\\
&& x\in \{0,1\}^n,
\end{eqnarray*}
where $c\in \Bbb R^n$, $a_i\in \Bbb R^{n}$ and $b_i\in \Bbb R$ for $i=1,\ldots,m$, can be reformulated as the following special case of (UQ$_+$):
\begin{eqnarray*}
&\max&x^Tx+(c-e)^Tx\\
&{\rm s.~t.}& x^Tx+(a_i-e)^Tx\leq b_i,~i=1,\ldots,m,\\
&& 0\le x^Tx-e^Tx\le 0,\\
&& 0\leq x^Tx+(e_{j}-e)^Tx\leq 1,~j=1,\ldots,n,
\end{eqnarray*}
where $e\in \Bbb R^n$ is the vector of all ones and $e_{j}\in \Bbb R^n$ is the $j$-th column of the identity matrix. Since the $0$-$1$ Knapsack problem is already NP-hard, we see that the problem (UQ$_+$) remains NP-hard even when
 \begin{equation}
p\geq n+2. \label{as:0}
\end{equation}

Lagrangian duality plays a fundamental role in optimization, especially in the quadratic constrained quadratic programming (QCQP) problem. It provides an upper bound of the primal maximization problem, i.e., weak duality holds \cite{DB}. It is a common sense that strong duality holds for convex optimization under Slater condition, that is, the difference between
the optimal values of the primal and dual problems is zero \cite{DB}. Though, it fails for the general non-convex programs, strong duality may hold for some nonconvex (QCQP),  see for example, \cite{Ben14,Ben96,Bur13,Poly,Pong,WX,XWS,YZ} and references therein.

For the problem (UQ$_+$), Beck \cite{Be07} showed that the strong duality holds as long as
\begin{equation}
p\leq n-1.\label{as:1}
\end{equation}
Replacing the inner maximization problem of (\ref{Cheyb}) with its Lagrangian dual minimization problem, Beck \cite{Be07} established the convex relaxation of (\ref{Cheyb}). Under the above Assumption (\ref{as:1}), the convex relaxation is tight and hence the min-max problem (\ref{Cheyb}) is globally solved.
Later, Beck \cite{Be} further relaxed the above assumption (\ref{as:1}) to
\begin{equation}
p\leq n,\label{as:2}
\end{equation}
and then proved that the strong duality holds for the problem (UQ$_+$) if for each $i=1,\ldots,p$, exactly one of the following three cases occurs:
\begin{equation}
{\rm(i)}~ l_i=-\infty, ~{\rm(ii)}~ l_i=u_i, ~{\rm(iii)}~  u_i=+\infty. \label{pn}
\end{equation}
Consequently, the min-max problem (\ref{Cheyb}) is globally solved by Beck's convex relaxation under the weaker assumption (\ref{as:2}). However, when $p>n$, Beck's convex relaxation is no longer tight, and, to our knowledge,
the quality of Beck's bound is still unknown.

In this paper, we will make a thorough study on the problem (UQ). Our main contributions can be divided into the following three parts.

\textbf{(i)} We pay a revisit  to  the problem (UQ$_+$) without assuming exactly one of the three cases (\ref{pn}) holds (see Section 2).
Differently from the semidefinite programming (SDP) reformulations \cite{Be07,Be}, we establish the following second-order cone programming (SOCP) relaxation:
\begin{eqnarray}
{\rm (SOCP)}~~&\max_{x\in \Bbb R^n}& t+2b_{0}^Tx+d_{0} \label{RQ} \\
&{\rm s.~t.}& l_{i}\leq t+2b_{i}^Tx+d_{i}\leq u_{i},~i=1,\ldots,p,\nonumber\\
&& \left\|\left(\begin{array}{c}Q^{\frac{1}{2}}x\\ \frac{t-1}{2}\end{array}\right)\right\| \le  \frac{t+1}{2}
\label{soc}
\end{eqnarray}
and show that it provides an upper bound as tight as the primal SDP relaxation.
Using a simple and easy-to-understand proof, we show the equivalence between (UQ$_+$) and (SOCP) under the following assumption
\begin{equation}
{\rm either}~ {\rm rank}\left[b_{1},\ldots,b_{p}\right]\leq n-1~{\rm or}~ p=n,
\label{as:3}
\end{equation}
which is slightly weaker than Assumption (\ref{as:2}).
Besides, under the additional primal Slater assumption, we establish the strong duality for (UQ$_+$) and its Lagrangian dual.

\textbf{(ii)}We extend the above SOCP reformulation approach from (UQ$_+$) to the general non-convex (QCQP) (see Section 3). More precisely, we establish a new sufficient condition under which (QCQP) is hidden convex, i.e., it is equivalent to a convex programming problem.
As applications, we show that our new sufficient condition not only generalizes a few existing results for special (QCQP) but also partially gives answers to a few open questions in literature.

Our first corollary is that the trust region subproblem (TRS) \cite{C00} enjoys the strong duality, which is a well-known result \cite{F96,F04,RW}. Moreover, our new convex reformulation brings a new look at the hidden convexity of (TRS). Our sufficient condition for the strong duality of the extended trust region problem with linear inequality constraints coincides with Hsia and Sheu's condition \cite{HS} and improves Jeyakumar and Li's condition \cite{JL}.
As a further extension, we establish a new sufficient condition for the hidden convexity of variants of the extended trust region subproblem with not only linear constraints but also two-sided ball constraints  \cite{Bi14}.

For the weighted maximin dispersion problem \cite{Hai,Hai2}, our sufficient condition  answers the open question \cite{Hai} under what condition the corresponding semidefinite programming relaxation is tight.

As another application, we obtain the first sufficient condition to guarantee the hidden convexity of the general problem (UQ) without assuming $Q\succ 0$.

Finally, we consider the extended $p$-regularized subproblem ($p$-RS) \cite{G81,NP06,W07,C11} with additional linear constraints. A class of polynomial solvable cases of the extended ($p$-RS) has been studied in \cite{HSY} where $p=4$ and the number of linear constraints is fixed as a constant. The case $p\neq 4$ remains unknown \cite{HSY}.
Our sufficient condition identifies a class of polynomially solved cases of the extended ($p$-RS) with any $p>2$.

\textbf{(iii)}We propose an improved 
approximation algorithm for convex-constrained (UQ$_+$), i.e., $l_i=-\infty$ and $u_i<+\infty$ for $i=1,\ldots,p$ (see Section 4).
Actually, for the nonconvex quadratic optimization problem with ellipsoid constraints:
\begin{eqnarray*}
({\rm ECQP})~~&\max_{x\in \Bbb R^n}& f(x)= x^TAx+2b^Tx \\
&{\rm s.~t.}&\|F^kx+g^k\|^2\le 1,~k=1,\ldots,p,\nonumber
\end{eqnarray*}
where $A\in\Bbb {R}^{n\times n}$ is symmetric, $b\in \Bbb {R}^n$, $F^k\in\Bbb {R}^{r^{k}\times n}$,   $g^k\in \Bbb {R}^{r^{k}}$ and $\|g^{k}\|<1$ for $k=1,\ldots,p$,
Tseng \cite{Tseng} proposed an approximation algorithm based on the SDP relaxation to find a feasible solution $x$ of (ECQP) in polynomial time such that
\begin{equation}
f(x)\geq \left(\frac{1-\gamma}{\sqrt{p}+\gamma}\right)^2\cdot v({\rm SDP}),\label{tsb}
\end{equation}
where $\gamma:=\max_{k=1,\ldots,p}\|g^k\|$ and $v(\cdot)$ denotes the optimal value of problem ($\cdot$).
Very recently, Hsia et al. \cite{XWX} improved the approximation bound (\ref{tsb}) to
\begin{equation}
f(x)\geq \left(\frac{1-\gamma}{\sqrt{\widetilde{r}}+\gamma}\right)^2\cdot v({\rm SDP}),\label{x}
\end{equation}
where $\widetilde{r}=\min\bigg\{\left\lceil\frac{ \sqrt{8p+17}-3}{2}\right\rceil,n+1 \bigg\}$. Trivially, this approximation ratio (\ref{x}) holds for the convex-constrained (UQ$_+$), as it is a special case of (ECQP).
In this paper, based on our second-order cone programming relaxation,
we propose a new approximation algorithm, which finds a feasible solution $x$ in polynomial time such that
\begin{equation}
f(x)\geq \left(\frac{1-\widetilde{\gamma}}{\sqrt{2}+\widetilde{\gamma}}\right)^2\cdot v({\rm SOCP}),\label{new}
\end{equation}
where $\widetilde{\gamma}:=\max_{i=1,\ldots,p}\frac{\|Q^{-\frac{1}{2}}b_{i}\|}{\sqrt{u_{i}-d_{i}+\|Q^{-\frac{1}{2}}b_{i}\|^2}}$.
We notice that the approximation bound (\ref{new}) greatly improves (\ref{x}) as we can show that $v({\rm SOCP})=v({\rm SDP})$. It should be noted that, though this new approximation bound relies on the input data, it is independent of the numbers $p$ and $n$.

\textbf{(iv)}We extend the above new-developed approximation analysis for the convex-constrained (UQ$_+$) to the problem of finding the smallest ball enclosing a given intersection of balls, i.e.,
the min-max problem (\ref{Cheyb}) (see Section 5).  Replacing the inner maximization problem with its Lagrangian dual relaxation,
Beck proposed an efficient convex quadratic relaxation \cite{Be}, which globally solves (\ref{Cheyb}) when $p\le n$. But it could fail to find the global minimizer for the hard case $p>n$. To our knowledge, the quality of Beck's convex quadratic relaxation remain unknown as well as the quality of the feasible solution returned by Beck's approach. In this paper, we succeed in establishing the first approximation analysis. We remark that our approximation bound dependents only  on the distribution of the given  balls rather than  $p$ and $n$.

Throughout the paper, $v(\cdot)$ stands for the optimal value of problem $(\cdot)$.
Let $\Bbb R^n$ denote the $n$-dimensional vector space. Let $I_{r}$ the $r\times r$ identity matrix. For a matrix $Q$,  denote by $\mathcal{N}(Q)$ and $\mathcal{R}(Q)$ the null and range space of $Q$, respectively.
$\lambda_{\min}(Q)$ and $\lambda_{\max}(Q)$ stand for  the smallest and largest eigenvalues of $Q$, respectively. $Q\succ(\succeq)0$ means that $Q$ is positive (semi)definite. For $Q\succeq0$, denote by $Q^{\frac{1}{2}}$ the square root of $Q$, i.e., $Q=Q^{\frac{1}{2}}Q^{\frac{1}{2}}$. Moreover, if $Q\succ0$, $Q^{-\frac{1}{2}}$ is the inverse of $Q^{\frac{1}{2}}$.
The inner product of two matrices $A$ and $B$ is denoted by $A\bullet B={\rm Tr}(AB^T)=\sum_{i=1}^n\sum_{j=1}^na_{ij}b_{ij}$.
$\|\cdot\|$ denotes the Euclidean norm, i.e., $\|x\|=\sqrt{x^Tx}$. For a set $\Omega\subseteq\Bbb R^n$,
int$(\Omega)$ denotes the set of all the interior points in $\Omega$.
The linear subspace generated by $\{b_{1},\ldots,b_{p}\}$ is denoted by ${\rm span}\{b_{1},\ldots,b_{p}\}$. ${\rm dim}\{M\}$ denotes the dimension of the
linear subspace $M$.

\section{An SOCP Relaxation and Strong Duality for (UQ$_+$)}

In this section, we present a second-order cone programming relaxation (SOCP) for (UQ$_+$). It is a common sense that solving an SOCP is much easier than SDP.  Theoretically, though it is equivalent to the SDP relaxation, (SOCP) not only provides a simple and easy way to characterize the strong duality of (UQ$_+$), but also implies a slightly weaker sufficient condition for the strong duality. 

\subsection{SDP Relaxations}
We first consider the Lagrangian dual problem of (UQ$_+$).
Introducing $p$ free Lagrange multipliers $\lambda_1,\ldots,\lambda_p$ yields the Lagrangian function of (UQ$_+$):

\begin{eqnarray*}
L(x,\lambda)&=&f_0(x)+\sum_{i=1}^p\lambda_i^{+}(u_i-f_i(x))+\sum_{i=1}^p\lambda_i^{-}(f_i(x)-l_i)\\
&=&\left(1-\sum_{i=1}^p \lambda_{i}\right)x^TQx+
2\left(b_{0}-\sum_{i=1}^p \lambda_{i} b_{i}\right)^Tx-\sum_{i=1}^p\lambda_{i}d_i\\
&&+\sum_{i=1}^p(\lambda_{i}^{+}u_{i}-\lambda_{i}^{-}l_{i})
+d_{0},
\end{eqnarray*}
where $\lambda_{i}^{+}=\max\{\lambda_{i},0\}$, $\lambda_{i}^{-}=-\min\{\lambda_{i},0\}$, and note that $\lambda_{i}=\lambda_{i}^{+}-\lambda_{i}^{-}$ for $i=1,\ldots,p$.
Then the Lagrangian dual problem of (UQ$_+$) can be written as
\begin{eqnarray*}
{\rm (D)}~~\inf_{\lambda\in\Bbb R^p}~\left\{d(\lambda):=\sup_{x\in \Bbb R^n}L(x,\lambda)\right\}.
\end{eqnarray*}
By using Shor's relaxation scheme \cite{NS}, the dual problem (D) can be reformulated as the dual semidefinite programming (SDP) relaxation for (UQ$_+$):
\begin{eqnarray*}
{\rm (D\text{-}SDP)}~~& \inf & \tau+\sum_{i=1}^p(\lambda_{i}^{+}u_{i}-\lambda_{i}^{-}l_{i})-\sum_{i=1}^p\lambda_id_i+d_0\\
&{\rm s.~t.}&
\sum_{i=1}^p\lambda_{i}\left(\begin{array}{cc}Q&b_{i}\\b_{i}^T&d_{i}
\end{array}\right)-\left(\begin{array}{cc}Q&b_{0}\\b_{0}^T&d_{0}
\end{array}\right)+\left(\begin{array}{cc}0&0\\0&\tau
\end{array}\right)\succeq 0,\\
&&             \tau\in \Bbb R,~\lambda_i\in\Bbb R,~i=1,\ldots,p.
\end{eqnarray*}

The primal SDP relaxation for (UQ$_+$) can be directly obtained by
lifting $x\in \Bbb R^n$ to $Y=xx^T\in\Bbb R^{n\times n}$ and then relaxing $Y=xx^T$ to $Y\succeq xx^T$ (see \cite{Be07}):
\begin{eqnarray}
{\rm (P\text{-}SDP_1)}~~&\max &  \left(\begin{array}{cc}Q&b_{0}\\b_{0}^T&d_{0}
\end{array}\right)\bullet \left(\begin{array}{cc}Y&x\\x^T&1
\end{array}\right) \label{PSDP}\\
&{\rm s.~t.}& l_{i} \leq  \left(\begin{array}{cc}Q&b_{i}\\b_{i}^T&d_{i}
\end{array}\right)\bullet \left(\begin{array}{cc}Y&x\\x^T&1
\end{array}\right) \leq u_{i},~i=1,\ldots,p,\nonumber\\
&& \left(\begin{array}{cc}Y&x\\x^T&1
\end{array}\right)\succeq 0.\nonumber 
\end{eqnarray}
It is similar to verify that (D-SDP) is also the conic dual problem of (P-SDP$_1$).
Since $Q\succ 0$, there is a strictly feasible solution for (D-SDP). If we further assume that the Slater condition holds for (UQ$_+$), then (P-SDP$_1$) also has a strictly feasible solution. Therefore, strong duality holds for  (D-SDP) and (P-SDP$_1$), i.e., $v$(D-SDP)$=v$(P-SDP$_1$) (see, e.g., \cite{V}).

By introducing the invertible transformation $y=Q^{\frac{1}{2}}x$ and $c_{i}=Q^{-\frac{1}{2}}b_{i}$ for $i=0,1,\ldots,p$, we reformulate (UQ$_+$) as
\begin{eqnarray*}
{\rm (UQ_0)}~~&\max_{y\in \Bbb R^n}& y^Ty+2c_{0}^Ty+d_{0}\\
&{\rm s.~t.}& l_{i}-d_{i}\leq y^Ty+2c_{i}^Ty\leq u_{i}-d_{i},~i=1,\ldots,p.
\end{eqnarray*}
Lifting $y\in \Bbb R^n$ to $\left(\begin{array}{c}I_n\\y^T
\end{array}\right)\left(\begin{array}{cc}I_n&y
\end{array}\right)\in\Bbb R^{(n+1)\times (n+1)}$ and then relaxing $y^Ty$ to $z\in \Bbb R$, we obtain a new SDP relaxation for (UQ$_0$) \cite{Be}:
\begin{eqnarray*}
{\rm (P\text{-}SDP_2)}~~&\max & \left(\begin{array}{cc}0_{n\times n}&c_{0}\\c_{0}^T&1
\end{array}\right)\bullet \left(\begin{array}{cc}I_n&y\\y^T&z
\end{array}\right)+d_{0}\\
&{\rm s.~t.}& l_{i}-d_{i} \leq \left(\begin{array}{cc}0_{n\times n}&c_{i}\\c_{i}^T&1
\end{array}\right)\bullet \left(\begin{array}{cc}I_n&y\\y^T&z
\end{array}\right)\leq u_{i}-d_{i},~i=1,\ldots,p,\\
&& \left(\begin{array}{cc}I_n&y\\y^T&z
\end{array}\right)\succeq 0.
\end{eqnarray*}
Beck \cite{Be07,Be} used (P-SDP$_1$) and (P-SDP$_2$) to prove the strong duality under Assumptions (\ref{as:0}) and  (\ref{as:1}), respectively.

\subsection{An SOCP Relaxation}
To simplify the presentation, we first assume $d_{i}=0$ for $i=0,1,\ldots,p$. Otherwise, for $i=1,\ldots,p$, we replace $u_{i}$ and $l_{i}$ with $u_{i}-d_{i}$ and $l_{i}-d_{i}$, respectively. We also assume $Q=I_n$ as it has been made in (UQ$_0$).
Then, (UQ$_+$) is reduced to the following formulation:
\begin{eqnarray*}
{\rm (U)}~~&\max_{x\in \Bbb R^n}& f_0(x)=x^Tx+b_{0}^Tx\\
&{\rm s.~t.}& l_{i}\leq x^Tx+2b_{i}^Tx\leq u_{i},~i=1,\ldots,p.
\end{eqnarray*}
The next result is the main theorem of the section.
\begin{thm}\label{thm1}
Under the Assumption (\ref{as:3}),
the problem ${\rm(U)}$ is equivalent to
the following SOCP relaxation:
\begin{eqnarray}
{\rm (S)}~~&\max_{x\in \Bbb R^n}& t+b_{0}^Tx\nonumber\\
&{\rm s.~t.}& l_{i}\leq t+2b_{i}^Tx\leq u_{i},~i=1,\ldots,p,\label{r11}\\
&& 
\left\|\left(\begin{array}{c} x\\ \frac{t-1}{2}\end{array}\right)\right\| \le  \frac{t+1}{2},\label{xxt}
\end{eqnarray}
in the sense that  $x^*$  globally solves ${\rm (U)}$ if and only if $(x^*,t^*):=(x^*,x^{*T}x^*)$ globally  solves ${\rm (S)}$.
\end{thm}
\begin{proof}
We first notice that the constraint (\ref{xxt}) is equivalent to
\begin{equation}
x^Tx\le t.\label{con}
\end{equation}
Suppose $v({\rm S})=+\infty$, it follows from (\ref{con}) that for any unbounded feasible solution of (S), say $(x,t)$, we have $t\rightarrow +\infty$. Since
\[
t+2b_{i}^Tx \ge t-2\sqrt{b_i^Tb_i}\sqrt{x^Tx}=t-2\|b_i\|\sqrt{t}  \rightarrow +\infty,
\]
we have $u_{i}=+\infty$, for $i=1,\ldots,p$.  Thus, we obtain $v({\rm U})=+\infty$.
Now we can assume $v({\rm S})<+\infty$ and let $(x^*, t^*)$ be an optimal solution of (S). If $(x^*)^Tx^*=t^*$, then $x^*$ is a feasible solution of (U) and $v({\rm S})=f_0(x^*)\le v({\rm U})$. On the other hand, as (S) is a relaxation of (U), we always have $v({\rm S})\ge v({\rm U})$. Therefore, it holds that $v({\rm S})= v({\rm U})$, which completes the proof.
 In the remaining, we assume $(x^*)^Tx^*<t^*$ and consider the following two cases:
 \begin{itemize}
\item[(i)] Suppose ${\rm rank}\left[b_{1},\ldots,b_{p}\right]\leq n-1$. Let $m$ be  the largest number such that there are $m$ indices $i_1,\ldots,i_m\in \{1,\ldots,p\}$  satisfying that ${\rm rank}\left[b_{i_1},\ldots,b_{i_m}\right]=m$ and the $i_k$-th  constraint in (\ref{r11}) is active at $(x^*, t^*)$ for $k=1,\ldots,m$. According to the definitions, we have $m\le n-1$ and
\[
t^*+2b_{i_k}^Tx^*=\delta_{i_k},~k=1,\ldots,m,
\]
where $\delta_{i_k}=l_{i_k}$ or $u_{i_k}$ for $k=1,\ldots,m$. Then, there is a nonzero vector $\widetilde{b}\in \Bbb R^n$ such that $b_{i_k}^T\widetilde{b}=0$ for $k=1,\ldots,m$. Since
\[
\widetilde{x}(\varepsilon):= x^*+\varepsilon \widetilde{b}
\]
remains feasible for (S) when  $|\varepsilon|$ is sufficiently small, we have $\widetilde{b}^Tb_0=0$. It follows that $\widetilde{x}(\varepsilon)$ remains optimal for small $|\varepsilon|$.  Define
\begin{eqnarray*}
E&=&\{\varepsilon\in\Bbb R\mid~l_{i}\leq t^*+2b_{i}^T\widetilde{x}(\varepsilon)\leq u_{i},~i=1,\ldots,p\},\\
J&=&\{j\in\{1,\ldots,p\}\setminus\{i_1,\ldots, i_m\}\mid~
{\rm rank}\left[b_{i_1},\ldots,b_{i_m},b_j\right]=m+1,~b_{j}^T\widetilde{b}\neq 0\}.
\end{eqnarray*}
If $E$ is unbounded, there is an $\varepsilon$ such that $\widetilde{x}(\varepsilon)^T\widetilde{x}(\varepsilon)=t^*$ and hence completes the proof.
Otherwise, it is sufficient to assume $J\neq \emptyset$ and $E$ is bounded,
since $J=\emptyset$ implies the unboundedness of $E$.
In this case, there is a $j_0\in J$ and $\varepsilon_0$ such that
$\widetilde{x}(\varepsilon_0)$ is feasible  for (S) and the $j_0$-th constraint in (\ref{r11}) becomes active at $(x^*, t^*)$, we obtain a contradiction against the definition of $m$.

\item[(ii)] Suppose $p=n$. According to Case (i), it is sufficient to consider the case ${\rm rank}\left[b_{1},\ldots,b_{p}\right] =n$ and all the $n$ constraints (\ref{r11}) are active at $(x^*, t^*)$, i.e.,
\begin{equation}
Bx^*=\delta-t^*e, \label{lineq}
\end{equation}
where $\delta=(\delta_i)\in\Bbb R^n$ with $\delta_{i}=l_{i}$ or $u_{i}$, $e\in \Bbb R^n$ is the vectors of all ones, the coefficient matrix $B=[2b_{1}, 2b_{2}, \ldots, 2b_{n}]^T$ is nonsingular.   Define
\[
\widetilde{x}(t)= B^{-1}(\delta-te).
\]
Then (S) is equivalently reduced to the following one-dimensional optimization problem:
\[
{\rm (S')}~~ \max_{t\ge \widetilde{x}(t)^T\widetilde{x}(t)} t+ b_{0}^T\widetilde{x}(t)
=  \max_{g(t)\le 0} \left\{ b_{0}^TB^{-1}\delta +t(1-b_{0}^TB^{-1}e)\right\},
\]
where $g(t)=(\delta-te)^TB^{-T}B^{-1}(\delta-te)-t$ is a strictly convex quadratic function.
It follows from the facts
 \[
 \widetilde{x}(t^*)=x^*,~g(t^*)=x^{*T}x^*-t^*<0,~\lim_{t\rightarrow +\infty}g(t)=+\infty
 \]
that the equation $g(t)=0$ has two roots $\widetilde{t}_1<\widetilde{t}_2$ and
\[
\{t\mid~g(t)\le 0\}= \{t\mid~ \widetilde{t}_1 \le t\le \widetilde{t}_2\}.
\]
Since  the objective function of (${\rm S'}$) is linear,
either $\widetilde{t}_1$ or $\widetilde{t}_2$ is an optimal solution of (${\rm S'}$). But both satisfy $t=\widetilde{x}(t)^T\widetilde{x}(t)$.  The proof is complete.
\end{itemize}
\end{proof}

The following example illustrates that Assumption (\ref{as:3}) seems to be tight and can not be expected for an improvement.
\begin{exam}
Consider the following instance of (U)
\begin{eqnarray*}
 &\max& x^2\\
&{\rm s.~t.}& 1\leq x^2+2x\leq3,\\
&&  -1\leq x^2-2x\leq3,
\end{eqnarray*}
where $n=1,~p=2$, and ${\rm rank}[b_{1},b_{2}]=1$. Assumption (\ref{as:3}) fails to hold.
It is easy to verify that the optimal value is $1$. However,
the SOCP relaxation (S) 
\begin{eqnarray*}
 &\max& t\\
&{\rm s.~t.}& 1\leq t+2x\leq3,\\
&&  -1\leq t-2x\leq3,\\
&& \left\|\left(\begin{array}{c} x\\ \frac{t-1}{2}\end{array}\right)\right\| \le  \frac{t+1}{2}
\end{eqnarray*}
has an optimal value $3$.
\end{exam}

The strong duality for (UQ$_+$) was first  established in \cite{Be07} under Assumption  (\ref{as:1}) and then strengthened in \cite{Be} by assuming that (\ref{as:2}) holds and exactly one of the three cases (\ref{as:3}) occurs. At the end of this section, we show that, following Theorem \ref{thm4.2}, the second-order cone programming relaxation (\ref{RQ}) is as tight as the primal SDP relaxation (\ref{PSDP}) and the strong duality immediately holds for (UQ$_+$)  under the slightly weaker Assumption (\ref{as:3}).

\begin{thm}\label{thm4.2}
Suppose the Slater condition holds for $({\rm UQ}_+)$ and $v({\rm UQ}_+)<+\infty$, we have
\begin{equation}
v({\rm SOCP})=v({\rm P\text{-}SDP_1})=v({\rm D\text{-}SDP}).\label{sdp-soc}
\end{equation}
 Moreover, the problem ${\rm (UQ}_+)$ enjoys the strong duality under the further assumption (\ref{as:3}), i.e., $v({\rm UQ}_+)=v({\rm D\text{-}SDP})$.
\end{thm}
\begin{proof}
As proved at beginning of Theorem \ref{thm1}, if  $v{\rm (UQ}_+)<+\infty$,
then $v({\rm SOCP})<+\infty$. 
Replacing (\ref{soc}) with the equivalent convex quadratic constraint
$x^TQx\le t$, we can reformulate  (SOCP) as a convex QCQP, denoted by (CQCQP). Then we have
\[
v({\rm SOCP})=v({\rm CQCQP}).
\]
It is not difficult to verify that the Lagrangian dual of (CQCQP) is exactly the same as that of ${\rm (UQ}_+)$, i.e., $({\rm D\text{-}SDP})$. Then, the Slater condition guarantees that  (CQCQP) itself enjoys the strong duality, i.e.,
\[
v({\rm CQCQP})=v({\rm D\text{-}SDP}),
 \]
see for example, Proposition 6.5.6 in \cite{Bert}. Since the Slater condition holds for $({\rm P\text{-}SDP_1})$ and its conic dual $({\rm D\text{-}SDP})$, strong duality holds for $({\rm P\text{-}SDP_1})$ and $({\rm D\text{-}SDP})$ \cite{V}, i.e.,
\[
v({\rm P\text{-}SDP_1})= v({\rm D\text{-}SDP}).
\]
Therefore, we obtain the equalites (\ref{sdp-soc}).
On the other hand, since linear transformation does not change the strong duality, it follows from Theorem \ref{thm1} that
\[
v{\rm (UQ}_+)=v({\rm SOCP})
\] 
under the further assumption (\ref{as:3}).
Thus, the proof of the second part is complete according to (\ref{sdp-soc}).
\end{proof}

\section{Sufficient Conditions for Hidden Convexity of the General (QCQP)}

Now we extend the above convexity approach for the uniform quadratic optimization problem (UQ$_+$) to the general quadratic
constrained quadratic programming (QCQP):
\begin{eqnarray}
{\rm (QCQP)}~~&\min & g_{0}(x) ~~({\rm or}~\max-g_{0}(x)) \label{qcqp:1}\\
&{\rm s.~t.}& l_i\le g_{i}(x)\leq u_i,~i=1,\ldots,p,\label{qcqp:2}
\end{eqnarray}
where the functions $g_{i}(x): \Bbb R^n\rightarrow \Bbb R~(i=0,1,\ldots,p)$ are given by
 \begin{eqnarray}
g_{i}(x)=\sum_{j=1}^m a_{ij}x^TQ_{j}x+2b_{i}^Tx+c_{i},\label{f}
 \end{eqnarray}
with $a_{ij}\in \{-1,0,1\}$, $Q_{j}=Q^T_{j}\succeq 0$, $b_{i}\in \Bbb R^n$ and $c_{i}\in\Bbb R$, for $i=0,1,\ldots,p$ and $j=1,\ldots,m$. (QCQP) is the simplest type of nonconvex nonlinear programming. It has many applications, see \cite{A97,Bao11} and references therein.
In this section, we establish a new sufficient condition under which (QCQP) is equivalent to its convex relaxation. Our new result not only generalizes
a few existing results for some special (QCQP) but also partially gives answers to a few open questions in literature.

\subsection{One-Sided QCQP}
In this subsection, we study the one-sided (QCQP), denoted by (QCQP$_1$), which is  a relatively easy case of (QCQP) with $l_i=-\infty$ for $i=1,\ldots,p$.
Our first main result is as follows.
\begin{thm}\label{mainthm}
Suppose
\begin{eqnarray}
\max_{j\in J}\left\{{\rm dim}\left\{{\rm span}\{b_{1},\ldots,b_{p}\}\bigcup \mathcal N(Q_{j})\bigcup_{i\neq j} \mathcal R(Q_{i})\right\}\right\}\leq n-1,\label{c}
\end{eqnarray}
the problem $({\rm QCQP_1})$ is equivalent to the following convex QCQP relaxation:
\begin{eqnarray}
{\rm (CR)}~~&\min & \sum_{j\not\in J}x^TQ_{j}x+\sum_{j\in J}a_{0j}t_{j}+2b_{0}^Tx+c_{0} \nonumber\\
&{\rm s.~t.}& \sum_{j\not\in J}x^TQ_{j}x+\sum_{j\in J}a_{ij}t_{j}+2b_{i}^Tx+c_{i}\leq 0,~i=1,\ldots,p,\nonumber\\
&& \left\|\left(\begin{array}{c} Q_j^{\frac{1}{2}}x \\ \frac{t_j-1}{2}\end{array}\right)\right\| \le  \frac{t_j+1}{2},~{j\in J},\label{cr1}
\end{eqnarray}
in the sense
that  $x^*$  globally solves $({\rm QCQP_1})$ if and only if $(x^*,t_j^*(j\in J)):=(x^*,x^{*T}Q_jx^*(j\in J))$ globally  solves $({\rm CR})$,
under the assumption $v({\rm CR})>-\infty$, where
$$
J=\bigcup\left\{j\in\{1,\ldots,m\}\mid~there~is~ an~ i\in \{0,\ldots,p\}~ such~ that~ a_{ij}=-1\right\}.
$$
\end{thm}
\begin{proof}
Let $(x^*, t_{j}^*(j\in J))$ be an optimal solution of (CR) since $v({\rm CR})>-\infty$ and the feasible region of (CR) is closed. Note that (\ref{cr1}) is equivalent to 
\[
{x}^TQ_{j}{x}\le t_{j}.
\]
Suppose ${x}^TQ_{j}{x}= t_{j}$ for all $j\in J$, the proof is complete. Now,
we assume that there is a $j_0\in J$ such that ${x^*}^TQ_{j_0}{x^*}<t_{j_0}^*$.

Let $(\cdot)^{\bot}$ be the orthogonal complement of the subspace $(\cdot)$. Notice that $\mathcal N^{\bot}(Q)=\mathcal R(Q)$ and $(A\cup B)^{\bot}=A^{\bot}\cap B^{\bot}$. Then, (\ref{c}) implies that
\[
\min_{j\in J}\left\{{\rm dim}\left\{\left({\rm span}\{b_{1},\ldots,b_{p}\}\right)^{\bot}\bigcap \mathcal R(Q_{j})\bigcap_{i\neq j}\mathcal N(Q_{i})\right\}\right\}\geq 1.
\]
Therefore, there is an $x_0\in \left({\rm span}\{b_{1},\ldots,b_{p}\}\right)^{\bot}\bigcap \mathcal R(Q_{j_0})\bigcap_{i\neq j_0}\mathcal N(Q_{i})$ and $x_0\neq 0$.
Since $Q_{j_{0}}\succeq 0$, we must have $x_0^TQ_{j_{0}}x_0>0$. Otherwise, $x_0^TQ_{j_{0}}x_0=0$ and then $Q_{j_{0}}x_0=0$. It follows that $x_0\in \mathcal R(Q_{j_0})\bigcap \mathcal N(Q_{j_0})=\{0\}$, which is contradiction.
According to the fact ${x^*}^TQ_{j_0}{x^*}<t_{j_0}^*$, the following equation
\[
(x^*+\alpha x_{0})^TQ_{j_{0}}(x^*+\alpha x_{0})=t_{j_{0}}^*
\]
has a solution, denoted by $\alpha_0$. Define
\[
\widetilde{x}:= x^*+\alpha_0 x_{0}.
\]
We can see that $(\widetilde{x},t_j^*(j\in J))$
is  feasible for $({\rm QCQP_1})$ with the same objective value as $(x^*,t_j^*(j\in J))$. That is, $(\widetilde{x},t_j^*(j\in J))$ is also an optimal solution of $({\rm QCQP_1})$.  The proof is complete by noting that
$\widetilde{x}^TQ_{j_0}\widetilde{x}=t_{j_0}^*$.
\end{proof}

Next, we show that Theorem \ref{mainthm} has many applications.

We first consider the well-known trust region subproblem \cite{C00,F96,F04,RW}:
\begin{eqnarray*}
(\rm TRS)~~&\min & x^TAx+2b^Tx\\
&{\rm s.~t.}& \|x\|^2\leq1,
\end{eqnarray*}
where $A=A^T\in\Bbb R^{n\times n}$ and $b\in \Bbb R^n$. If $A\succeq 0$, (TRS) is already a convex programming problem. Otherwise, we have:
\begin{cor}\label{cor:1}
Suppose $A\not\succeq 0$, the nonconvex
(TRS) is equivalent to the following ball constrained convex quadratic programming relaxation:
\begin{eqnarray*}
&\min& x^T(A-\lambda_{\min}(A)I_{n})x+2b^Tx+\lambda_{\min}(A)\\
&{\rm s.~t.}& \|x\|^2\leq1.
\end{eqnarray*}
\end{cor}
\begin{proof}
We can recast (TRS) as the following formulation:
\begin{eqnarray*}
&\min & x^T(A-\lambda_{\min}(A)I_{n})x+2b^Tx+\lambda_{\min}(A)x^Tx\\
&{\rm s.~t.}& x^Tx\leq1.
\end{eqnarray*}
According to Theorem \ref{mainthm}, it is equivalent to the following relaxation
\begin{eqnarray}
{\rm (TRSR)}~~&\min & x^T(A-\lambda_{\min}(A)I_{n})x+2b^Tx+\lambda_{\min}(A)t \nonumber\\
&{\rm s.~t.}& t\le 1, ~x^Tx\leq t, \nonumber
\end{eqnarray}
since $v({\rm TRSR})>-\infty$ and the assumption (\ref{c}) reduces to
\[
{\rm dim}\left\{\mathcal R(A-\lambda_{\min}(A)I_{n})\right\}\leq n-1,
\]
which is always true.
As we have assumed $A\not\succeq 0$, it holds that $\lambda_{\min}(A)<0$. Thus, for any optimal solution of the above problem, denoted by $(x^*,t^*)$, we always have $t^*=1$, which completes the proof.
\end{proof}

As a further extension, we study (TRS) with additional linear inequality constraints \cite{HS,JL,sz}:
\begin{eqnarray}
{\rm (ETRS)}~~&\min_{x\in \Bbb R^n}& x^TAx+a^Tx\nonumber\\
&{\rm s.~t.}&  \|x-x_{0}\|^2\leq u,\nonumber\\
&& b_{i}^Tx\leq \beta_{i},~i=1,\ldots,p,\nonumber
\end{eqnarray}
where $A=A^T\in\Bbb R^{n\times n}$, $a,~b_{i},~x_{0}\in \Bbb R^n$ and $u,~\beta_{i}\in \Bbb R$.
Again, we assume $A\not\succeq 0$, i.e., $\lambda_{\min}(A)<0$.
\begin{cor}
Suppose $A\not\succeq 0$ and
\begin{equation}
{\rm dim}\left\{{\rm span}\{b_{1},\ldots,b_{p}\}\bigcup\mathcal R(A-\lambda_{\min}(A)I_{n})\right\}\leq n-1,\label{etrs}
\end{equation}
(ETRS) is equivalent to the following linear and ball constrained convex quadratic programming relaxation:
 \begin{eqnarray*}
&\min & x^T(A-\lambda_{\min}(A) I_{n})x+a^Tx+\lambda_{\min}(A) (u+2x^Tx_{0}-\|x_{0}\|^2)\nonumber\\
&{\rm s.~t.}&  \|x-x_{0}\|^2\leq u,\label{o}\\
&& b_{i}^Tx\leq \beta_{i},~i=1,\ldots,p.\nonumber
\end{eqnarray*}
 \end{cor}
\begin{proof}
Similar to the proof of Corollary \ref{cor:1}, it is sufficient to show that (ETRS) is equivalent to
 \begin{eqnarray*}
{\rm (ER)}~~&\min & x^T(A-\lambda_{\min}(A) I_{n})x+a^Tx+\lambda_{\min}(A)t\nonumber\\
&{\rm s.~t.}&  t-2x^Tx_{0}+\|x_{0}\|^2\leq u,\label{o}\\
&& b_{i}^Tx\leq \beta_{i},~i=1,\ldots,p,\nonumber\\
&& x^Tx\leq t.\nonumber
\end{eqnarray*}
Firstly, it is easy to verify that the assumption (\ref{etrs}) coincides with (\ref{c}). To show $v{\rm (ER)}>-\infty$, it is sufficient to prove the feasible region of ${\rm (ER)}$ is bounded. According to the Cauchy-Schwartz inequality, for any feasible solution of (ER), denoted by $(x,t)$, we have
\[
\|x\|^2\leq t\leq u-\|x_{0}\|^2 +2x^Tx_{0} \leq u-\|x_{0}\|^2 +2\|x_{0}\|\|x\|.
\]
Therefore, both $\|x\|$ and $t$ are bounded. Now, the equivalence between (ETRS) and (ER) directly follows from Theorem \ref{mainthm}.
\end{proof}
\begin{remark}
The assumption (\ref{etrs}) was first proposed by
Hsia and Sheu \cite{HS} to guarantee that $v{\rm (ETRS)}$ is equal to the optimal objective value of the primal SDP relaxation of ${\rm (ETRS)}$. It improved the following assumption due to Jeyakumar and Li \cite{JL}
\[
{\rm dim}\left\{\mathcal N(A-\lambda_{\min}(A)I_n)\right\}\ge{\rm dim}\left\{{\rm span}\left\{b_1,\ldots,b_p\right\}\right\}+1.
\]
\end{remark}

Now we focus on the weighted max-min dispersion problem \cite{Hai,Hai2}:
\[
\max_{x\in \Omega}\min_{i=1,\ldots,p}\omega_{i}\|x-z_{i}\|^2,
\]
where $\Omega\subseteq \Bbb R^n$ is closed, $z_{1},\ldots, z_{p}$ in $\Bbb R^n$ are given points and $\omega_{1},\ldots, \omega_{p}$ are positive weights.
When $\Omega$ is a ball with a radius $r_0$ centering at $x_0$ , the max-min dispersion problem
can be reformulated as a special case of (QCQP):
\begin{eqnarray*}
(\rm WD)~~&\max&s\\
&{\rm s.~t.}&s\leq\omega_{i}\|x-z_{i}\|^2,~i=1,\ldots,p,\\
&&\|x-x_{0}\|^2\leq r_{0}^2.
\end{eqnarray*}
An approximation algorithm for (WD) was proposed in \cite{Hai,Hai2} based on the following SDP relaxation:
\begin{eqnarray}
(\rm WD\text{-}SDP)~~&\max& s\nonumber\\
&{\rm s.~t.}&s\leq\omega_{i} \left(\begin{array}{cc}I&-z_{i}\\-z_{i}^T&\|z_{i}\|^2\\\end{array}\right)\bullet \left(\begin{array}{cc}Y&x\\x^T&1\\\end{array}\right),~i=1,\ldots,p,\label{wd:sdp1}\\
&&  \left(\begin{array}{cc}I&-x_{0}\\-x_{0}^T&\|x_{0}\|^2\\\end{array}\right)\bullet  \left(\begin{array}{cc}Y&x\\x^T&1\\\end{array}\right) \leq r_{0}^2,\label{wd:sdp2}\\
&& \left(\begin{array}{cc}Y&x\\x^T&1\\\end{array}\right)\succeq 0,\label{wd:sdp3}
\end{eqnarray}
where $Y$ is relaxed from $xx^T$.
It was raised as an open question in \cite{Hai,Hai2} that
under what conditions the SDP relaxation of (WD)
is tight, i.e., $v({\rm WD})=v(\rm WD\text{-}SDP)$.
Here, we can quickly give an answer based on Theorem \ref {mainthm}.
\begin{cor}
Suppose
\begin{equation}
{\rm rank}\left[z_{1}-x_0,\ldots,z_{p}-x_0\right]\leq n-1,\label{wd}
\end{equation}
$({\rm WD})$ is equivalent to the following linear and ball constrained convex quadratic programming relaxation:
 \begin{eqnarray}
(\rm WR)~~&\max& s \nonumber\\
&{\rm s.~t.}&
s\leq\omega_{i}(r_0^2-2(z_{i}-x_0)^Ty+\|z_{i}-x_0\|^2),~i=1,\ldots,p,\label{wr:1}\\
&&\|y\|\leq r_{0}.\label{wr:2}
\end{eqnarray}
Moverover, under the assumption int$(\Omega)\neq\emptyset$, i.e., there exists an $\overline{x}$ such that $\|\overline{x}-x_0\|<r_0$, we always have $v({\rm WR})=v(\rm WD\text{-}SDP)$.
\end{cor}
\begin{proof}
Introducing $y:=x-x_0$, we see that (WD) is equivalent to
\begin{eqnarray*}
(\rm WD')~~&\max&s\\
&{\rm s.~t.}&s\leq\omega_{i}\|y-(z_{i}-x_0)\|^2,~i=1,\ldots,p,\\
&&\|y\|^2\leq r_{0}^2.
\end{eqnarray*}
For $({\rm WD'})$,
it is trivial to verify that Assumption (\ref{c}) reduces to (\ref{wd}).
According to Theorem \ref{mainthm}, $({\rm WD'})$ is equivalent to the following convex relaxation
 \begin{eqnarray*}
(\rm WR')~~&\max& s \nonumber\\
&{\rm s.~t.}&
s\leq\omega_{i}(t-2(z_{i}-x_0)^Ty+\|z_{i}-x_0\|^2),~i=1,\ldots,p,\label{wr:1}\\
&&t\leq r_{0}^2, \\
&& y^Ty\leq t,
\end{eqnarray*}
by noting that the feasible region of $(\rm WR')$ is clearly bounded.
Let $(y,t,s)$ be any feasible solution of $(\rm WR')$. $(y,r_{0}^2,s)$ remains feasible since $\omega_{i}>0$ for $i=1,\ldots,p$. Therefore, we always have
$v(\rm WR')=v(\rm WR)$.

The assumption int$(\Omega)\neq\emptyset$ implies that
the Slater condition holds for $({\rm WR'})$. Since $v{\rm (WR')}<+\infty$, according to
Proposition 6.5.6 in \cite{Bert}, strong duality holds for $({\rm WR'})$ and its Lagrangian dual.

Consider the SDP relaxation of $({\rm WR'})$:
\begin{eqnarray*}
(\rm WD\text{-}SDP')~~&\max& s\nonumber\\
&{\rm s.~t.}&s\leq\omega_{i} \left(\begin{array}{cc}I&x_0-z_{i}\\x_0^T-z_{i}^T&\|z_{i}-x_0\|^2\\\end{array}\right)\bullet \left(\begin{array}{cc}Z&y\\y^T&1\\\end{array}\right),~i=1,\ldots,p,\\
&&  \left(\begin{array}{cc}I&0\\0&0\\\end{array}\right)\bullet  \left(\begin{array}{cc}Z&y\\y^T&1\\\end{array}\right) \leq r_{0}^2, \\
&& \left(\begin{array}{cc}Z&y\\y^T&1\\\end{array}\right)\succeq 0.
\end{eqnarray*}
Let $(Z,y,s)$ be any feasible solution of $(\rm WD\text{-}SDP')$. It follows from Tr$(Z)=I\bullet Z\le r_{0}^2$ and $Z-yy^T\succeq 0$ that $\|y\|$ is bounded since
\[
y^Ty \le {\rm Tr}(Z)\leq r_{0}^2.
\]
Then, according to the first constraint of $(\rm WD\text{-}SDP')$, $s$ is bounded from above. Note that $Z\succeq 0$ implies that
\begin{eqnarray*}
&&Z_{ii}\ge 0,~i=1,\dots,n,\\
&&Z_{ii}Z_{jj}\ge Z_{ij}^2,~i,j=1,\dots,n.
\end{eqnarray*}
That is, each entry of $Z$ is bounded. Then the feasible region of $(\rm WD\text{-}SDP')$ is compact. Under the assumption int$(\Omega)\neq\emptyset$, the Slater condition holds for $(\rm WD\text{-}SDP')$. Therefore, strong duality holds for  $(\rm WD\text{-}SDP')$ and its conic dual, see for example \cite{V}.

It is not difficult to verify that the Lagrangian dual of $({\rm WR'})$ is the same as the conic dual of $(\rm WD\text{-}SDP')$. We have $v(\rm WD\text{-}SDP')=v({\rm WR'})$.

Let $(Y,x,s)$ is a feasible solution of $(\rm WD\text{-}SDP)$. We can verify that $(Z,y,s):=(Y-x_0x^T-xx_0^T+x_0x_0^T,x-x_0,s)$ is a feasible solution of $(\rm WD\text{-}SDP')$. On the other hand, for any feasible solution $(\rm WD\text{-}SDP')$, say, $(Z,y,s)$, $(Y,x,s):=(Z+x_0y^T+yx_0^T+x_0x_0^T,y+x_0,s))$ is a feasible solution of $(\rm WD\text{-}SDP)$. Therefore, $v(\rm WD\text{-}SDP')=v(\rm WD\text{-}SDP)$.

As a conclusion, we obtain
\[
v({\rm WR})=v({\rm WR'})=v(\rm WD\text{-}SDP')=v(\rm WD\text{-}SDP).
\]
\end{proof}

\subsection{Two-Sided QCQP}
Now we focus on the general two-sided (QCQP) (\ref{qcqp:1})-(\ref{qcqp:2}). Our second main result is as follows.
\begin{thm}\label{mainthm2}
Suppose
\begin{eqnarray}
\max_{j\in K}\left\{ {\rm dim}\left\{{\rm span}\left\{b_{1},\ldots,b_{p}\right\}\bigcup \mathcal N(Q_{j})\bigcup_{i\neq j} \mathcal R(Q_{i})\right\}\right\}\leq n-1,\label{cc}
\end{eqnarray}
the problem (QCQP) is equivalent to the following linear and second-order cone constrained convex quadratic programming relaxation:
\begin{eqnarray*}
{\rm (CR')}~~&\min & \sum_{j\not\in K}x^TQ_{j}x+\sum_{j\in K}a_{0j}t_{j}+2b_{0}^Tx+c_{0}\\
&{\rm s.~t.}& l_{i}\leq\sum_{j\in K}a_{ij}t_{j}+2b_{i}^Tx+c_{i}\leq u_{i},~i=1,\ldots,p,\\
&& \left\|\left(\begin{array}{c} Q_j^{\frac{1}{2}}x \\ \frac{t_j-1}{2}\end{array}\right)\right\| \le  \frac{t_j+1}{2},~{j\in J},~{j\in K},
\end{eqnarray*}
in the sense
that  $x^*$  globally solves $({\rm QCQP})$ if and only if $(x^*,t_j^*(j\in K)):=(x^*,x^{*T}Q_jx^*(j\in K))$ globally  solves $({\rm CR'})$,
under the assumption $v({\rm CR'})>-\infty$, where
\[
K=\bigcup  \left\{j\in\{1,\ldots,m\}\mid~a_{0j}=-1~or~ a_{ij}\neq 0~for~some~i\in\{1,\ldots,p\}\right\}.
\]
\end{thm}

Our first application of Theorem \ref{mainthm2} is for the general problem (UQ) rather than (UQ$_+$). We first consider the positive semidefinite case.
\begin{cor}
Suppose $Q\succeq0$  and
\[
{\rm rank}\left[{b}_{1},\ldots,{b}_{p}\right]\leq {\rm rank}(Q)-1,
\]
The problem ${\rm(UQ)}$ is equivalent to ${\rm(SOCP)}$ (\ref{RQ}) in the sense that  $x^*$  globally solves ${\rm(UQ)}$ if and only if $(x^*,t^*):=(x^*,x^{*T}x^*)$ globally  solves ${\rm(SOCP)}$ (\ref{RQ}).
\end{cor}
Then we study the indefinite case of (UQ). Let $Q=Q_{1}-Q_{2}$ be a decomposition such that $Q_1\succeq 0$, $Q_2\succeq 0$, rank$(Q_1)=r_1$, and rank$(Q_2)=r_2$. We can reformulate
the indefinite (UQ) as:
\begin{eqnarray*}
&\max & x^TQ_{1}x-x^TQ_{2}x+2b_{0}^Tx+d_{0}\\
&{\rm s.~t.}& l_{i}\leq x^TQ_{1}x-x^TQ_{2}x+2b_{i}^Tx\leq u_{i},~i=1,\ldots,p.
\end{eqnarray*}
Applying Theorem \ref{mainthm2} yields the following result.
\begin{cor}
Suppose
\[
{\rm rank}\left[{b}_{1},\ldots,{b}_{p}\right]\leq \min\{r_{1},r_{2}\}-1,
\]
the indefinite ${\rm(UQ)}$ is equivalent to the following  second-order cone programming relaxation:
\begin{eqnarray*}
{\rm (SOCP')}~~&\max_{x\in \Bbb R^n}& t_{1}-t_{2}+2b_{0}^Tx+d_{0}\\
&{\rm s.~t.}& l_{i}\leq t_{1}-t_{2}+2b_{i}^Tx\leq u_{i},~i=1,\ldots,p,\label{r15}\\
&& \left\|\left(\begin{array}{c} Q_1^{\frac{1}{2}}x\\ \frac{t_1-1}{2}\end{array}\right)\right\| \le  \frac{t_1+1}{2},\\
&& \left\|\left(\begin{array}{c} Q_2^{\frac{1}{2}}x\\ \frac{t_2-1}{2}\end{array}\right)\right\| \le  \frac{t_2+1}{2},
\end{eqnarray*}
in the sense that  $x^*$  globally solves ${\rm(UQ)}$ if and only if
\[
(x^*,t_1^*,t_2^*):=(x^*,x^{*T}Q_1x^*,x^{*T}Q_2x^*)
\]
globally  solves ${\rm (SOCP')}$.
\end{cor}
The two-sided ball constrained trust region subproblem was first proposed by Stern and Wolkowicz \cite{SW}, see also \cite{Ben96,Pong}:
\begin{eqnarray*}
({\rm TTRS})~~&\min & f(x)=\frac{1}{2}x^TAx+b^Tx  \\
&{\rm s.t.}&\alpha\le x^Tx\le \beta, 
\end{eqnarray*}
where $A=A^T\in\Bbb R^{n\times n}$, $b\in \Bbb R^n$ and $-\infty<\alpha<\beta<+\infty$. Directly applying Theorem \ref{mainthm2}, we have:
\begin{cor} 
The problem $({\rm TTRS})$ is equivalent to the following linear and second-order cone constrained convex quadratic programming relaxation:
\begin{eqnarray*}
({\rm CTTRS})~~&\min& \frac{1}{2}x^T(A-\lambda_{\min}(A) I)x+b^Tx+\lambda_{\min}(A) t \\
&{\rm s.t.}&\alpha\le t\le \beta, \\
&&\left\|\left(\begin{array}{c} x\\ \frac{t-1}{2}\end{array}\right)\right\| \le  \frac{t+1}{2},
\end{eqnarray*}
in the sense
that  $x^*$  globally solves $({\rm TTRS})$ if and only if $(x^*,t^*):=(x^*,x^{*T}x^*)$ globally  solves $({\rm CTTRS})$.
\end{cor}

Very recently, Bienstock and Michalka \cite{Bi14} studied the following variants of trust region subproblem:
\begin{eqnarray*}
(\rm {VTRS})~~&\min&x^TQx+c^Tx\\
&{\rm s.~t.}&\|x-\mu_{i}\|\leq r_{i},~i\in I,\\
&&\|x-\mu_{j}\|\geq r_{j},~j\in J,\\
&& x\in P,
\end{eqnarray*}
where $P=\{x\mid~a_{k}^Tx\leq b_{k},~k=1,\ldots,m\}$. 
 It was showed in \cite{Bi14} that (VTRS) is polynomially solvable as long as the number of faces of $P$ within the ellipsoids is polynomial.

Decomposing $Q=(Q-\lambda_{\min}(Q)I_{n})+\lambda_{\min}(Q)I_{n}$
and then applying Theorem \ref {mainthm2} yields the following new sufficient condition to guarantee the hidden convexity of (VTRS):
\begin{cor}
Suppose
\begin{eqnarray*}
{\rm dim} \left\{{\rm span}\{a_{1},\ldots,a_{m},\mu_i(i\in I),\mu_j(j\in J)\}\bigcup\mathcal R(Q-\lambda_{\min}(Q)I_{n})\right\}
\leq n-1,
\end{eqnarray*}
the problem (VTRS) is equivalent to the following linear and second-order cone constrained convex quadratic programming relaxation:
\begin{eqnarray*}
&\min &x^T(Q-\lambda_{\min}(Q)I_{n})x+c^Tx+\lambda_{\min}(Q)t\\
&{\rm s.~t.}& t-2\mu_{i}^Tx+\|\mu_{i}\|^2\leq r_{i},~i\in I,\\
&& t-2\mu_{j}^Tx+\|\mu_{j}\|^2\geq r_{j},~j\in J,\\
&& \left\|\left(\begin{array}{c} x\\ \frac{t-1}{2}\end{array}\right)\right\| \le  \frac{t+1}{2},\\
&& a_{k}^Tx\leq b_{k},~k=1,\ldots,m,
\end{eqnarray*}
in the sense that  $x^*$  globally solves (VTRS) if and only if $(x^*,t^*):=(x^*,x^{*T}x^*)$ globally  solves the convex relaxation problem.
\end{cor}

Finally, we consider the extended p-regularized subproblem (p-${\rm RS_{m}}$) with $m$ inequality constraints \cite{HSY}:
\begin{eqnarray*}
({\rm p}\text{-}{\rm RS_{m}})~~&\min_{x\in \Bbb R^n}&\frac{1}{2}x^TQx+c^Tx+\frac{\sigma}{p}\|x\|^p\\
&{\rm s.~t.}&l_{i}\leq a_{i}^Tx+b_{i}\leq u_{i},~i=1,\ldots,m,
\end{eqnarray*}
where $Q=Q^T\in\Bbb R^{n\times n}$, $c,~a_{i}\in \Bbb R^n$, $b_{i},~l_{i},~u_{i}\in \Bbb R$, $\sigma>0$ and $p>2$. Hsia et al. \cite{HSY} showed that if $m$ is fixed then
$({\rm p}\text{-}{\rm RS_{m}})$ is polynomially solvable  when $p=4$. The complexity for the general $p>2$ remains unknown, see \cite{HSY}.

Again, based on the matrix decomposition $Q=(Q-\lambda_{\min}(Q)I_{n})+\lambda_{\min}(Q)I_{n}$,
we apply Theorem \ref {mainthm2} to establish a new polynomially solvable class of $({\rm p}\text{-}{\rm RS_{m}})$ for $p>2$:

\begin{cor}
Under the condition that
\begin{eqnarray*}
{\rm dim}\left\{{\rm span}\{a_{1},\ldots,a_{m}\}\bigcup\mathcal R(Q-\lambda_{\min}(Q)I_{n})\right\}\leq n-1,
\end{eqnarray*}
the problem (p-${\rm RS_{m}}$) is equivalent to the following linear and second-order cone constrained convex quadratic programming relaxation:
\begin{eqnarray*}
~~&\min &\frac{1}{2}x^T(Q-\lambda_{\min}(Q)I_{n})x+c^Tx+\frac{\sigma}{p}t^\frac{p}{2}+\lambda_{\min}t\\
&{\rm s.~t.}&l_{i}\leq a_{i}^Tx+b_{i}\leq u_{i},~i=1,\ldots,m,\\
&& \left\|\left(\begin{array}{c} x\\ \frac{t-1}{2}\end{array}\right)\right\| \le  \frac{t+1}{2}.
\end{eqnarray*}
in the sense that  $x^*$  globally solves (p-${\rm RS_{m}}$) if and only if $(x^*,t^*):=(x^*,x^{*T}x^*)$ globally  solves the convex relaxation problem.
\end{cor}

\section{Approximation Algorithm for Convex-Constrained (UQ$_+$)}

As shown in Theorem \ref{thm4.2}, $v({\rm UQ_+})=v({\rm SOCP})$ under Assumption (\ref{as:3}). Suppose (\ref{as:3}) does not hold, it is not difficult to find examples satisfying $v({\rm UQ_+})<v({\rm SOCP})$, see \cite{Be07}. Thus,
it is natural to ask what is the quality of $v({\rm SOCP})$. In this section, we answer this question for the convex-constrained case of (UQ$_+$) where  $l_{i}=-\infty,~i=1,\ldots,p$. More precisely, a new approximation algorithm based on (SOCP) is proposed for solving (UQ$_+$).  It should be noted that the approximation algorithms for the general nonconvex quadratic optimization problem with ellipsoid constraints \cite{Tseng,XWX} work  well for (UQ$_+$). However, as we see below, for (UQ$_+$), the existing approximation bounds can be greatly improved. Moreover, the existing approximation algorithms \cite{Tseng,XWX} are based on SDP relaxation, while our new algorithm is based on SOCP relaxation.

\begin{thm}\label{mainthm3}
Under the assumption that the origin $0$ is in the interior of the feasible region of  (UQ$_+$), $d_0=0$, $l_{i}=-\infty$ and $u_{i}<+\infty$ for $i=1,\ldots,p$, we can find a feasible solution $x$ in polynomial time satisfying
\begin{equation}
v({\rm SOCP})\ge v({\rm UQ_+})\ge f_0(x)\geq \left(\frac{1-\gamma}{\sqrt{2}+\gamma}\right)^2\cdot v({\rm SOCP})\ge 0,\label{new:0}
\end{equation}
where $\gamma=\max_{i=1,\ldots,p}\frac{\|Q^{-\frac{1}{2}}b_{i}\|}{\sqrt{u_{i}-d_{i}+\|Q^{-\frac{1}{2}}b_{i}\|^2}}$.
\end{thm}
\begin{proof}
The assumptions $0$ is in the interior of the feasible region and $d_0=0$ imply that $v({\rm SOCP})\ge v({\rm UQ_+})\ge 0$. Under the assumption $u_{i}<+\infty$ for $i=1,\ldots,p$, the feasible region of (SOCP) is compact and hence $v({\rm SOCP})<+\infty$. Let  $(x^*,t^*)$ be an optimal solution of (SOCP). It is sufficient to
assume ${x^*}^TQx^*<t^*$ since otherwise, $v({\rm UQ_+})=v({\rm SOCP})$.
Now we have
\[
t^*+2b_{0}^Tx^*=v({\rm SOCP}),~
t^*+2b_{i}^Tx^*\leq u_{i}-d_{i},~
x^{*T}Qx^*< t^*,~i=1,\ldots,p.
\]
Let $y\in \Bbb R^n$ satisfy
\[
x^{*T}Qx^*+y^TQy=t^*.
\]
Then we have  $y\neq 0$, $y^TQy>0$ and
\begin{eqnarray}
\left(\alpha^2+1\right)\left(x^{*T}Qx^*+y^TQy+2b_{0}^Tx^*\right)=\left(\alpha^2+1\right)
\cdot v({\rm SOCP}).\label{m1}
\end{eqnarray}
Since ${x^*}^TQx^*+2b_{0}^Tx^*< v({\rm SOCP})$ and $y^TQy>0$,  there exits a real value $\alpha>0$ such that
\begin{eqnarray}
(x^*+\alpha y)^TQ(x^*+\alpha y)+2b_{0}^T(x^*+\alpha y)=v({\rm SOCP}).\label{m2}
\end{eqnarray}
The equation (\ref{m1}) minus the equation (\ref{m2}) equals
\begin{eqnarray}
&&(\alpha x^*-y)^TQ(\alpha x^*-y)+2\alpha b_{0}^T(\alpha x^*-y)
=\alpha^2\cdot v({\rm SOCP}).\label{m4}
\end{eqnarray}
Define
\[
s_{1}=\frac{x^*+\alpha y}{\sqrt{1+\alpha^2}},~s_{2}=\frac{\alpha x^*-y}{\sqrt{1+\alpha^2}}.
\]
Then (\ref{m2}) and (\ref{m4}) can be recast as
\begin{eqnarray}
&&s_{1}^TQs_{1}+\frac{2b_{0}^Ts_{1}}{\sqrt{1+\alpha^2}}=\frac{v({\rm SOCP})}{1+\alpha^2},\label{n1}\\
&&s_{2}^TQs_{2}+\frac{2\alpha b_{0}^Ts_{2}}{\sqrt{1+\alpha^2}}=\frac{\alpha^2v({\rm SOCP})}{1+\alpha^2}.\label{n2}
\end{eqnarray}
Since
\[
s_{1}^TQs_{1}+s_{2}^TQs_{2}=t^*,~ x^*=\frac{s_{1}}{\sqrt{1+\alpha^2}}+\frac{\alpha s_{2}}{\sqrt{1+\alpha^2}},
\]
it follows from $t^*+2b_{i}^Tx^*\leq u_{i}-d_{i}$ that
\[
s_{1}^TQs_{1}+\frac{2b_{i}^Ts_{1}}{\sqrt{1+\alpha^2}}+s_{2}^TQs_{2}+\frac{2\alpha b_{i}^Ts_{2}}{\sqrt{1+\alpha^2}}\leq u_{i}-d_{i},~i=1,\ldots,p,
\]
or equivalently,
\begin{eqnarray*}
\frac{\bigg\|Q^{\frac{1}{2}}s_{1}+\frac{Q^{-\frac{1}{2}}b_{i}}{\sqrt{1+\alpha^2}}\bigg\|^2}{u_{i}-d_{i}+\|Q^{-\frac{1}{2}}b_{i}\|^2}+\frac{\bigg\|Q^{\frac{1}{2}}s_{2}+
\frac{\alpha Q^{-\frac{1}{2}}b_{i}}{\sqrt{1+\alpha^2}}\bigg\|^2}{u_{i}-d_{i}+\|Q^{-\frac{1}{2}}b_{i}\|^2}\leq 1, ~i=1,\ldots,p.
\end{eqnarray*}
Then we have
\begin{eqnarray*}
&\min& \left\{ \max_{i=1,\ldots,p}\frac{\left(1+\alpha^2\right)\bigg\|Q^{\frac{1}{2}}s_{1}+\frac{Q^{-\frac{1}{2}}b_{i}}{\sqrt{1+\alpha^2}}\bigg\|^2}{u_{i}-d_{i}+\|Q^{-\frac{1}{2}}b_{i}\|^2},
\max_{i=1,\ldots,p}\frac{\left(\frac{1+\alpha^2}{\alpha^2}\right)\bigg\|Q^{\frac{1}{2}}s_{2}+\frac{\alpha Q^{-\frac{1}{2}}b_{i}}{\sqrt{1+\alpha^2}}\bigg\|^2}{u_{i}-d_{i}+\|Q^{-\frac{1}{2}}b_{i}\|^2}\right\}\\
&\le&
\min  \left\{  1+\alpha^2,~\frac{1+\alpha^2}{\alpha^2}\right\}\\
&=&
\left\{\begin{array}{ll}1+\alpha^2,
& {\rm if}~|\alpha|\leq 1,\\
1+\frac{1}{\alpha^2},
& {\rm otherwise},\end{array}\right.\\
&\leq&2.
\end{eqnarray*}
Therefore, there is an index $\overline{j}\in \{1,2\}$ such that
\begin{equation}
\frac{\|Q^{\frac{1}{2}}s_{\overline{j}}/t_{\overline{j}}+Q^{-\frac{1}{2}}b_{i}\|}{\sqrt{u_{i}-d_{i}+\|Q^{-\frac{1}{2}}b_{i}\|^2}}
\le \sqrt{2},~i=1,\ldots,p.\label{16}
\end{equation}
where $t_{1}=\frac{1}{\sqrt{1+\alpha^2}}$, $t_{2}=\frac{\alpha}{\sqrt{1+\alpha^2}}$.
Define
\begin{eqnarray*}
\overline{x}&:=&\left\{\begin{array}{ll}s_{\overline{j}}/t_{\overline{j}},
& {\rm if}~b_{0}^Ts_{\overline{j}}/t_{\overline{j}}\geq 0,\\
-s_{\overline{j}}/t_{\overline{j}},
& {\rm otherwise},\end{array}\right.\\
\overline{\tau}&:=&\max\bigg\{\tau\in[0,1]:~f_i(\tau \overline{x})\le u_i,~i=1,\ldots,p\bigg\}\\
&~=&\max\bigg\{\tau\in[0,1]:~\frac{\|\tau Q^{\frac{1}{2}}\overline{x}+Q^{-\frac{1}{2}}b_{i}\|}{\sqrt{u_{i}-d_{i}+\|Q^{-\frac{1}{2}}b_{i}\|^2}}\le 1,~i=1,\ldots,p\bigg\}.
\end{eqnarray*}
According to \eqref{16},
it holds that
\begin{eqnarray*}
&&\frac{\|Q^{\frac{1}{2}}\overline{x}+Q^{-\frac{1}{2}}b_{i}\|}{\sqrt{u_{i}-d_{i}+\|Q^{-\frac{1}{2}}b_{i}\|^2}}
\\&&\le~~ \max\left\{
\frac{\|Q^{\frac{1}{2}}s_{\overline{j}}/t_{\overline{j}}+Q^{-\frac{1}{2}}b_{i}\|}{\sqrt{u_{i}-d_{i}+\|Q^{-\frac{1}{2}}b_{i}\|^2}},
\frac{\|-(Q^{\frac{1}{2}}s_{\overline{j}}/t_{\overline{j}}+Q^{-\frac{1}{2}}b_{i})+2Q^{-\frac{1}{2}}b_{i}\|}{\sqrt{u_{i}-d_{i}+\|Q^{-\frac{1}{2}}b_{i}\|^2}}
\right\}\\
&&\le~~\max\left\{\sqrt{2},~ \sqrt{2}+\frac{2\|Q^{-\frac{1}{2}}b_{i}\|}{\sqrt{u_{i}-d_{i}+\|Q^{-\frac{1}{2}}b_{i}\|^2}}
\right\}\\
&&=~~\sqrt{2}+\frac{2\|Q^{-\frac{1}{2}}b_{i}\|}{\sqrt{u_{i}-d_{i}+\|Q^{-\frac{1}{2}}b_{i}\|^2}}.
\end{eqnarray*}
Therefore, for any $\tau\in[0,1]$, we obtain
\begin{eqnarray*}
&&\frac{\|\tau Q^{\frac{1}{2}}\overline{x}+Q^{-\frac{1}{2}}b_{i}\|}{\sqrt{u_{i}-d_{i}+\|Q^{-\frac{1}{2}}b_{i}\|^2}}
\\&&=~~\frac{\|\tau (Q^{\frac{1}{2}}\overline{x}+Q^{-\frac{1}{2}}b_{i})+(1-\tau)Q^{-\frac{1}{2}}b_{i}\|}{\sqrt{u_{i}-d_{i}+\|Q^{-\frac{1}{2}}b_{i}\|^2}}
\\
&&\leq~~\tau\left(\sqrt{2}+\frac{2\|Q^{-\frac{1}{2}}b_{i}\|}{\sqrt{u_{i}-d_{i}+\|Q^{-\frac{1}{2}}b_{i}\|^2}}\right)+(1-\tau)\frac{\|Q^{-\frac{1}{2}}b_{i}\|}{\sqrt{u_{i}-d_{i}+\|Q^{-\frac{1}{2}}b_{i}\|^2}}.
\end{eqnarray*}
Since the origin $0$ is in the interior of the feasible region of  (UQ$_+$), we have $u_i-d_i>0$ for $i=1,\ldots,p$. It follows that $\frac{\|Q^{-\frac{1}{2}}b_{i}\|}{\sqrt{u_{i}-d_{i}+\|Q^{-\frac{1}{2}}b_{i}\|^2}}<1$.
Now, for $i=1,\ldots,p$, we have
\[
\frac{\|\tau Q^{\frac{1}{2}}\overline{x}+Q^{-\frac{1}{2}}b_{i}\|}{\sqrt{u_{i}-d_{i}+\|Q^{-\frac{1}{2}}b_{i}\|^2}}\leq 1,
\]
or equivalently, $f_i(\tau\overline{x})\le u_i$,
as long as
\[
\tau\leq \left(1-\frac{\|Q^{-\frac{1}{2}}b_{i}\|}{\sqrt{u_{i}-d_{i}+\|Q^{-\frac{1}{2}}b_{i}\|^2}}\right)\bigg/\left(\sqrt{2}+\frac{\|Q^{-\frac{1}{2}}b_{i}\|}{\sqrt{u_{i}-d_{i}+\|Q^{-\frac{1}{2}}b_{i}\|^2}}\right).
\]
According to the definition of $\overline{\tau}$, we obtain
\begin{eqnarray*}
\overline{\tau}&\geq&\min_{i=1,\ldots,p}{\left(1-\frac{\|Q^{-\frac{1}{2}}b_{i}\|}{\sqrt{u_{i}-d_{i}+\|Q^{-\frac{1}{2}}b_{i}\|^2}}\right)\bigg/
\left(\sqrt{2}+\frac{\|Q^{-\frac{1}{2}}b_{i}\|}{\sqrt{u_{i}-d_{i}+\|Q^{-\frac{1}{2}}b_{i}\|^2}}\right)}\\
&=&\frac{1-\max_{i=1,\ldots,p}\frac{\|Q^{-\frac{1}{2}}b_{i}\|}{\sqrt{u_{i}-d_{i}+\|Q^{-\frac{1}{2}}b_{i}\|^2}}}{\sqrt{2}+\max_{i=1,\ldots,p}\frac{\|Q^{-\frac{1}{2}}b_{i}\|}{\sqrt{u_{i}-d_{i}+\|Q^{-\frac{1}{2}}b_{i}\|^2}}},
\end{eqnarray*}
where the equality holds as $h(\gamma)=(1-\gamma)/(\sqrt{2}+\gamma)$ is a decreasing function for
$\gamma\in [0,1)$.

Now, we conclude that
\begin{align}
f_{0}(\overline{\tau}\overline{x}) =&\overline{\tau}^2\overline{x}^TQ\overline{x}+2\overline{\tau} b_{0}^T\overline{x}\nonumber\\
 \geq &\overline{\tau}^2\overline{x}^TQ\overline{x}+2\overline{\tau}^2 b_{0}^T\overline{x}\label{add0}\\
 \geq &\overline{\tau}^2\overline{x}^TQ\overline{x}+2\overline{\tau}^2 b_{0}^Ts_{\overline{j}}/t_{\overline{j}}\label{add1}\\
 =& \overline{\tau}^2({s_{\overline{j}}}^TQs_{\overline{j}}+2t_{\overline{j}}b_{0}^Ts_{\overline{j}})/{t_{\overline{j}}}^2\nonumber\\
 = &\overline{\tau}^2\cdot v({\rm SOCP})\label{add2},
\end{align}
where the inequality \eqref{add0} follows since $b_0^T\overline{x}\ge 0$ and $\overline{\tau}\geq \overline{\tau}^2$ ($\overline{\tau}\in[0,1]$), the inequality \eqref{add1} holds as $b_0^T\overline{x}\geq b_{0}^Ts_{\overline{j}}/t_{\overline{j}}$,
and the equality \eqref{add2} follows from (\ref {n1}) and (\ref{n2}).
\end{proof}
\begin{remark}
To satisfy the assumption that $0$ is in the interior of the feasible region of  (UQ$_+$), it is sufficient to find any one of the interior points  of (UQ$_+$) and then translate the origin  there.
\end{remark}

\section{Approximation Algorithm for Finding the Chebyshev Center of the Intersection of Balls}

As an application of (UQ$_+$), Beck \cite{Be07,Be} studied the problem finding the Chebyshev center of the intersection of given balls, i.e., the min-max problem (CC) (\ref{Cheyb}).
Replacing the inner maximization problem with its Lagrangian dual, which is the following minimization SDP problem:
\begin{eqnarray*}
({\rm SDP}(z))~~&\min ~~& t\\
&{\rm s.~t.}& \left(\begin{array}{cc}(-1+\sum_{i=1}^p\lambda_{i})I_{n}&z-\sum_{i=1}^p\lambda_{i}a_{i}\\
z^T-\sum_{i=1}^p\lambda_{i}a_{i}^T&t+\sum_{i=1}^p\lambda_{i}(\|a_{i}\|^2-r_{i}^2)
\end{array}\right)\succeq 0,\\
&&\lambda_{i}\in \Bbb R_{+},~i=1,\ldots,p,
\end{eqnarray*}
Beck \cite{Be07} reduced the min-max problem (CC) to a double minimization problem, which turns out to be convex programming problem:
\begin{eqnarray*}
({\rm DCC})~~&&\min_{z}~ v({\rm SDP}(z))+\|z\|^2 \\
&=&\min_{t,\lambda,z} ~ t+\|z\|^2\\
&&{\rm s.~t.}~ \left(\begin{array}{cc}(-1+\sum_{i=1}^p\lambda_{i})I_{n}&z-\sum_{i=1}^p\lambda_{i}a_{i}\\
z^T-\sum_{i=1}^p\lambda_{i}a_{i}^T&t+\sum_{i=1}^p\lambda_{i}(\|a_{i}\|^2-r_{i}^2)
\end{array}\right)\succeq 0,\\
&&\lambda_{i}\in \Bbb R_{+},~i=1,\ldots,p.
\end{eqnarray*}
Moreover, by noting that the optimal objective value of above convex SDP problem can be attained at a solution $(t^*,\lambda^*,z^*)$ satisfying $\sum_{i=1}^p\lambda_{i}^*=1$ and $z^*=\sum_{i=1}^p\lambda^*_{i}a_{i}$, Beck \cite{Be07} showed that  (DCC) is further equivalent to the following standard convex quadratic programming problem:
\begin{eqnarray*}
&\min_{\lambda} & \sum_{i=1}^p\lambda_{i}(r_{i}^2-\|a_{i}\|^2)+\left\|\sum_{i=1}^p\lambda_{i}a_{i}\right\|^2\\
&{\rm s.~t.}&
 \sum_{i=1}^p\lambda_{i}=1,~\lambda_{i}\ge 0,~i=1,\ldots,p.
\end{eqnarray*}
Based on the strong duality theory for (UQ$_+$), Beck \cite{Be} showed $v({\rm CC})=v({\rm DCC})$ when $p\leq n$. For the hard case $p>n$, solving (DCC) always yields a heuristic Chebyshev center $\overline{z}$.
To our knowledge, it remains unknown what is the quality of $v({\rm DCC})$ and $\overline{z}$.
In this section, we establish the first approximation ratio between $v({\rm CC})$ and $v({\rm DCC})$. Moreover, the quality of the solution $\overline{z}$ returned by Beck's approach is also guaranteed.

\begin{thm}
Under the assumption that ${\rm int}(\Omega)\neq \emptyset$, we can find a solution $\overline{z}$ in polynomial time satisfying
\begin{eqnarray}
v({\rm DCC})\ge \max_{x\in \Omega}\|x-\overline{z}\|^2 \ge v({\rm CC})
\ge
\left(\frac{1-\gamma}{\sqrt{2}+\gamma}\right)^2\cdot v({\rm DCC}),\label{ratio}
\end{eqnarray}
where $\gamma$ is equal to the optimal objective value of the following convex programming problem:
\begin{equation}
\min_{x\in\Bbb R^n}\max_{i=1,\ldots,p}\frac{\|x-a_{i}\|}{r_{i}}. \label{gam0}
\end{equation}
Moreover, let $d_{\max}=\max_{i,j=1,\ldots,p}\|a_{i}-a_{j}\|$ and $r_{\min}=\min_{i=1,\ldots,p}r_i$, we have
\begin{equation}
\gamma\le \sqrt{\frac{n}{2(n+1)}}\cdot \frac{d_{\max} }{r_{\min}}
<\frac{d_{\max} }{\sqrt{2}~ r_{\min}}.
\label{gam1}
\end{equation}
\end{thm}
\begin{proof}
According to Theorem \ref{thm4.2},  $v({\rm SDP}(z))=v({\rm SOCP}(z))$, where ${\rm SOCP}(z)$ is the second-order cone programming relaxation proposed in Section 2:
\begin{eqnarray*}
{\rm SOCP}(z)~~&\max& t-2z^Tx\\
&{\rm s.~t.}&  t-2a_{i}^Tx+\|a_{i}\|^2\leq {r_{i}}^2,~i=1,\ldots,p,\\
&& \left\|\left(\begin{array}{c} x\\ \frac{t-1}{2}\end{array}\right)\right\| \le  \frac{t+1}{2}.
\end{eqnarray*}
Suppose $0\in {\rm int}(\Omega)$, according to Theorem \ref{mainthm3}, for any $z\in\Bbb R^n$, we have
\[
v({\rm SDP}(z))\ge
 \max_{x\in \Omega}\{\|x\|^2-2x^Tz\}\ge \tau^2\cdot v({\rm SDP}(z)),
\]
where \[
\tau=\frac{1-\max_{i=1,\ldots,p}\frac{\|a_{i}\|}{r_{i}}}{\sqrt{2}+\max_{i=1,\ldots,p}\frac{\|a_{i}\|}
 {r_{i}}}.
 \]
Since $0\in {\rm int}(\Omega)$, we have $\|a_{i}\|<r_{i}$ for $i=1,\ldots,p$. That is, $\tau>0$. Notice that it is trivial to see
 $\tau<1$. Therefore, $\tau^2<1$ and then it holds that
\begin{eqnarray*}
\min_{z}\{v({\rm SDP}(z))+\|z\|^2\}\ge \min_{z}\max_{x\in \Omega}\|x-z\|^2 \ge
\tau^2\min_{z}\{v({\rm SDP}(z))+\|z\|^2\},
\end{eqnarray*}
or equivalently,
\begin{equation}
v({\rm DCC}) \ge v({\rm CC})
\ge \tau^2\cdot v({\rm DCC}).\label{e:1}
\end{equation}
According to the definition of $({\rm SOCP}(z))$, we have
\[
v({\rm SOCP}(z)) \ge \|x\|^2-2z^Tx,~\forall~x\in\Omega.
\]
Let $\overline{z}$ be an optimal solution of (DCC). Since $v({\rm SDP}(z))=v({\rm SOCP}(z))$, it holds that
\begin{equation}
v({\rm SDP}(\overline{z}))+\|\overline{z}\|^2\ge \|x-\overline{z}\|^2,~\forall~x\in\Omega.\label{e:2}
\end{equation}
Therefore, combining (\ref{e:1}) and (\ref{e:2}) yields
\begin{equation}
v({\rm DCC})\ge \max_{x\in \Omega}\|x-\overline{z}\|^2 \ge v({\rm CC})\ge \tau^2\cdot v({\rm DCC}).\label{e:3}
\end{equation}
Let $x_0$ be any interior point of $\Omega$, i.e., $ \|x_0-a_{i}\|< r_{i}$, $i=1,\ldots, p$.
Define
\begin{eqnarray*}
({\rm CC'})~~\min_{z}\max_{\widetilde{x}\in \Omega(x_0)}\|\widetilde{x}-(z-x_0)\|^2,
\end{eqnarray*}
where
$\Omega(x_0)=\{\widetilde{x}\in \Bbb R^n\mid~ \|\widetilde{x}-(a_{i}-x_0)\|^2\leq r_{i}^2, ~i=1,\ldots, p\}$. Then it is trivial to see that $0\in {\rm int}(\Omega(x_0))$ and $v({\rm CC})=v({\rm CC'})$.

We write the second-order cone programming relaxation for the inner maximization problem of $({\rm CC'})$ as follows:
\begin{eqnarray*}
{\rm SOCP'}(z)~~&\max& \widetilde{t}-2(z-x_0)^T\widetilde{x}\\
&{\rm s.~t.}&  \widetilde{t}-2(a_{i}-x_0)^T\widetilde{x}+\|a_{i}-x_0\|^2\leq {r_{i}}^2,~i=1,\ldots,p,\\
&&  \left\|\left(\begin{array}{c} \widetilde{x}\\ \frac{\widetilde{t}-1}{2}\end{array}\right)\right\| \le  \frac{\widetilde{t}+1}{2}.
\end{eqnarray*}
Let $(x^*,t^*)$ be an optimal solution of ${\rm SOCP}(z)$. Define
\[
\widetilde{x}:=x^*-x_0,~\widetilde{t}:=t^*-2x_0^Tx^*+x_0^Tx_0.
\]
We can verify  that $(\widetilde{x},\widetilde{t})$ is a feasible solution of ${\rm SOCP'}(z)$ and hence
\begin{equation}
v({\rm SOCP'}(z))\geq \widetilde{t}-2(z-x_0)^T\widetilde{x}
=v({\rm SOCP}(z))+\|z\|^2-\|z-x_0\|^2.\label{e:4}
\end{equation}
On the other hand,
let $(\widetilde{x}^*,\widetilde{t}^*)$ be an optimal solution of ${\rm SOCP'}(z)$. Define
\[
x:=\widetilde{x}^*+x_0,~t:=\widetilde{t}^*+2x_0^T\widetilde{x}^*-x_0^Tx_0.
\]
We can verify  that $(x,t)$ is a feasible solution of ${\rm SOCP}(z)$ and hence
\begin{equation}
v({\rm SOCP}(z))\geq t-2z^Tx
=v({\rm SOCP'}(z))-\|z\|^2+\|z-x_0\|^2.\label{e:5}
\end{equation}
Combining (\ref{e:4}) and (\ref{e:5}) yields
\begin{equation}
v({\rm SOCP}(z))+\|z\|^2
=v({\rm SOCP'}(z))+\|z-x_0\|^2.\label{e:6}
\end{equation}
Denote by ${\rm SDP'}(z)$ the SDP relaxation  for the inner maximization problem of $({\rm CC'})$. According to Theorem \ref{thm4.2}, we have $v({\rm SDP'}(z))=v({\rm SOCP'}(z))$. Then, the equation (\ref{e:6}) becomes
\begin{equation}
v({\rm SDP}(z))+\|z\|^2
=v({\rm SDP'}(z))+\|z-x_0\|^2.\label{e:7}
\end{equation}
Similar to $({\rm DCC})$, we define
\[
({\rm DCC'})~~\min_{z}~ v({\rm SDP'}(z))+\|z-x_0\|^2.
\]
Then the equation (\ref{e:7}) implies that
\[
v({\rm DCC'})=v({\rm DCC}).
\]
Since $0\in {\rm int}(\Omega(x_0))$, according to (\ref{e:3}), we have
\[
v({\rm DCC'})\ge \max_{\widetilde{x}\in \Omega(x_0)}\|\widetilde{x}-(\widetilde{z}-x_0)\|^2 \ge v({\rm CC'})\ge \tau(x_0)^2\cdot v({\rm DCC'}),
\]
or equivalently,
\begin{equation}
v({\rm DCC})\ge \max_{x\in \Omega}\|x-\widetilde{z}\|^2 \ge v({\rm CC})\ge \tau(x_0)^2\cdot v({\rm DCC}), \label{e:m}
\end{equation}
where $\widetilde{z}$ is an optimal solution of $({\rm DCC'})$ and
\[
\tau(x_0)=\frac{1-\max_{i=1,\ldots,p}\frac{\|x_0-a_{i}\|}{r_{i}}}{\sqrt{2}+\max_{i=1,\ldots,p}\frac{\|x_0-a_{i}\|}
 {r_{i}}}.
\]
Since the inequality (\ref{e:m}) holds for any $x_0\in {\rm int}(\Omega)$, we can choose $x_0$ to maximize the lower bound $\tau(x_0)$ in ${\rm int}(\Omega)$, that is,
\[
\tau(x_0)^2=\left(\frac{1-\gamma}{\sqrt{2}+\gamma}\right)^2,
\]
where
\begin{eqnarray}
\gamma &=&\inf_{x_0\in {\rm int}(\Omega)}\max_{i=1,\ldots,p}\frac{\|x_0-a_{i}\|}{r_{i}}\nonumber\\
&=&\min_{x_0\in \Omega}\max_{i=1,\ldots,p}\frac{\|x_0-a_{i}\|}{r_{i}}.\label{gam}
\end{eqnarray}
Since $\|x_0-a_{i}\|\le r_i$ for any $x\in \Omega$, we have $\gamma\le 1$.  Let $x^*$ be an optimal solution of (\ref{gam0}).  Notice that the optimal objective value of (\ref{gam0}) provides a lower bound of the problem (\ref{gam}). Then, we have $\max_{i=1,\ldots,p}\frac{\|x^*-a_{i}\|}{r_{i}}\le 1$, that is, $x^*\in \Omega$. Therefore, the problems (\ref{gam0}) and (\ref{gam}) are actually equivalent.

According to the definition of $r_{\min}$, we obtain
\[
\gamma \le \min_{x_0\in \Bbb R^n}\max_{i=1,\ldots,p}\frac{\|x_0-a_{i}\|}{r_{\min}} =
\frac{1}{r_{\min}}\cdot\min_{x_0\in \Bbb R^n}\max_{i=1,\ldots,p}\|x_0-a_{i}\|.
\]
Since the objective function of the above inner maximization problem is concave with respect to $a_i$, we have
\[
\gamma\le \frac{1}{r_{\min}}\cdot\min_{x_0\in \Bbb R^n}\max_{a\in C}\|x_0-a\|
\]
where $C$ is the  convex hull of the points $a_{1},\ldots,a_{p}$.
According to Example 3.3.6 in \cite{DB}, we have
\[
\min_{x_0\in \Bbb R^n}\max_{a\in C}\|x_0-a\| \le d_{\max}\sqrt{\frac{n}{2(n+1)}}
\]
and then the proof is complete.
\end{proof}

\section{Conclusion}
In this paper, we have thoroughly studied the uniform quadratic optimization problem (UQ$_+$), which is a special quadratic constrained quadratic programming (QCQP) sharing the same positive definite Hessian matrix $Q$. We show that (UQ$_+$) is NP-hard when $p\ge n+2$, where $n$ and $p$ are the numbers of variables and constraints, respectively. We proposed a second-order cone programming (SOCP) relaxation for (UQ$_+$) and showed that it is tight under a new assumption, which is slightly weaker than the existing assumption $p\le n$.  Now the complexity of (UQ$_+$) remains unknown only when $p=n+1$. Then, we
extended the SOCP relaxation approach to (QCQP) and established a new sufficient condition under which (QCQP) is hidden convex. As further applications, this new condition not only generalized some well-known results (for example, the trust region subproblem and extensions,  (UQ) with positive semidefinite $Q$ and indefinite $Q$), but also partially gave answers to a few open questions (for example, the max-min dispersion problem and the extended p-regularized subproblem). For convex constrained (UQ$_+$), we proposed an imroved approximation algorithm based on our SOCP relaxation. Moreover, the new approximation bound is independent of $p$ and $n$.
As a further application, we succeeded in establishing the first approximation analysis for Beck's convex approach to find the Chebyshev center of the intersection of $p$ balls. Again, the approximation bound dependents only  on the distribution of the given  balls rather than  $p$ and $n$.
The future work may include more applications of our sufficient condition for general (QCQP) and further improvement of our approximation analysis.

\end{document}